\documentclass[12pt]{amsart}

\usepackage{amsmath}
\usepackage{amsfonts}
\usepackage{amssymb}
\usepackage{amsthm}
\usepackage{stmaryrd}
\usepackage{hyperref}
\usepackage{enumerate}
\usepackage{cleveref}
\usepackage{mathrsfs}
\usepackage{geometry}

\newtheorem{theorem}{Theorem}[section]
\newtheorem*{theorem*}{Theorem}
\newtheorem{lemma}[theorem]{Lemma}
\newtheorem{proposition}[theorem]{Proposition}
\newtheorem{corollary}[theorem]{Corollary}

\Crefname{conjecture}{Conjecture}{Conjectures}

\theoremstyle{remark}
\newtheorem*{remark*}{Remark}
\newtheorem{remark}[theorem]{Remark}

\theoremstyle{plain}
\newtheorem{example}[theorem]{Example}
\theoremstyle{plain}

\newcommand{\N}{\mathbb{N}}

\newcommand{\Z}{\mathbb{Z}}

\newcommand{\Q}{\mathbb{Q}}
\newcommand{\R}{\mathbb{R}}
\newcommand{\C}{\mathbb{C}}

\newcommand{{\D}}{\delta}
\newcommand{\F}{\mathbb{F}}

\newcommand{\calO}{\mathcal{O}}

\newcommand{\Stab}{\operatorname{Stab}}
\newcommand{\eps}{\varepsilon}

\newcommand{\SL}{\operatorname{SL}}

\newcommand{\SLZ}{\SL_2(\Z)}
\newcommand{\abcd}[4]{\left(\begin{smallmatrix} #1 & #2 \\ #3 & #4 \end{smallmatrix}\right)}

\newcommand{\calM}{\mathcal{M}}

\newcommand{\HH}{\mathfrak{H}}
\newcommand{\Zed}{\mathfrak Z}
\newcommand{\Zhat}{\widehat\Zed}
\newcommand{\Norm}{\mathsf N}
\renewcommand{\span}{\operatorname{span}}
\newcommand{\Eich}{\mathcal{E}}
\newcommand{\fraka}{\mathfrak{a}}
\newcommand{\calJ}{\mathcal{J}}
\newcommand{\ch}{\operatorname{ch}}
\newcommand{\orb}{\operatorname{orb}}

\newcommand{\bbM}{\mathbb{M}}

\numberwithin{equation}{section}
\numberwithin{table}{section}

\author{Lea Beneish and Michael H. Mertens}
\address{Department of Mathematics and Computer Science, Emory University, 400 Dowman Drive, 30322 Atlanta, GA } 
\email{lea.beneish@emory.edu}
\address{Max-Planck-Institut f\"ur Mathematik, Vivatsgasse 7, 53111 Bonn, Germany} 
\email{mhmertens@mpim-bonn.mpg.de}

\title[Weierstrass mock modular forms and VOAs]{On Weierstrass mock modular forms and a dimension formula for certain vertex operator algebras}


\begin{document}
\begin{abstract}
Using techniques from the theory of mock modular forms and harmonic Maa{\ss} forms, especially Weierstrass mock modular forms, we establish several dimension formulas for certain holomorphic, strongly rational vertex operator algebras, complementing previous work by van Ekeren, M\"oller, and Scheithauer.
\end{abstract}
\maketitle

\section{Introduction and statement of results}

Much of the motivation to study vertex operator algebras (VOAs) originates from work explaining observations originally made by McKay and Thompson \cite{Thompson1,Thompson2} on apparent relations between the representation theory of the Fischer-Griess Monster $\bbM$, the largest of the sporadic simple groups, and the modular $J$-function. Conway and Norton \cite{CN} later extended the original observations to the famous Monstrous moonshine conjecture. Frenkel, Lepowsky, and Meurman \cite{FLM84, FLM83, FLM} were the first to construct a vertex operator algebra $V^\natural$ whose character is the $J$-function and on which $\bbM$ acts as automorphisms. Borcherds \cite{Borcherds} later confirmed that the graded characters of $V^\natural$ coincide with certain Hauptmoduln identified by Conway and Norton \cite{CN}. Since then, VOAs and related structures have played a central role in various other instances of moonshine. The most illustrious cases of this include Norton's \emph{generalized moonshine} \cite{Norton,Norton2}, which has been proven in general by Carnahan \cite{Carnahan}. A first special case of this, now called \emph{Baby Monster moonshine} was established earlier by H\"ohn \cite{Hoehn}. Other instances of moonshine related to VOAs include \emph{Conway moonshine} \cite{DM} and various instances of \emph{umbral moonshine}  \cite{ACH,CD2,DD,DH}.

In recent new developments of moonshine, new connections to the arithmetic of elliptic curves have been discovered \cite{Beneish1,DMOnat,DMO}. In \cite{DMO} for instance, the existence of a representation for the sporadic O'Nan group has been established, which controls ranks and $p$-torsion in Selmer groups or Tate-Shafarevich groups of quadratic twists of certain elliptic curves (for a precise statement, see \cite[Theorems 1.3 and 1.4]{DMO}. See also \cite[Theorem 7.1]{D}.). Similar connections have been established for example in \cite{Beneish1} using quasimodular forms of weight $2$, along with a construction of the corresponding module as a VOA. In this work, we continue the mantra of connecting VOAs and arithmetic, but in a somewhat different vein.

In the context of classifying holomorphic, strongly rational VOAs of central charge $24$ (see \Cref{secVOA} for definitions of these terms) and proving the ``completeness'' of a list of $71$ Lie algebras devised by Schellekens \cite{Schellekens}, which is now known to contain every possible $V_1$-space of such a VOA by work of various people (see for instance the introduction of \cite{MSV} for references), van Ekeren, M\"oller, and Scheithauer \cite{MSV} find a dimension formula for orbifold VOAs of central charge $24$. A special case of their formula had previously been established by M\"oller in his thesis \cite{Mollerthesis} and for the sake of simplicity, we only give this special case here. The reader is referred to \Cref{secVOA} and the references given there for the relevant definitions.
\begin{theorem*}
Let $V$ be a holomorphic, strongly rational VOA additionally satisfying the positivity assumption and let $G=\langle g\rangle$ be a cyclic group of automorphisms of $V$ of order $N$ with $g$ of type $N\{0\}$. Denote by $V^G$ the fixed point VOA of $V$ under the action of $G$ and let $̃V^{\orb(g)}$ be the orbifold vertex operator algebra. Furthermore, assume that $V$ has central charge $c= 24$. Then for $N= 2,3,5,7,13$, we have the dimension formula
$$\dim V_1 + \dim ̃V^{\orb(g)}_1 = 24 + (N+ 1) \dim V^G_1-\frac{24}{N-1}\sum_{k=1}^{N-1} \sigma(N-k)\sum_{i\in\Z/N\Z \setminus\{0\}}\dim V(g^i)_{k/N},$$
where $\sigma(m)=\sum_{d\mid m} d$ denotes the usual divisor sum function.
\end{theorem*}
The proof of this result and of the extension of the result in \cite{MSV} for all $N$ such that the modular curve $X_0(N)$ has genus $0$, i.e. $N\in\{2,...,10,12,13,16,18,25\}$, is essentially obtained by writing the character $\ch_{V^G}$ explicitly in terms of the Hauptmodul for the group $\Gamma_0(N)$.

Our main result is an extension of the dimension formula in \cite{MSV} to levels $N$ where there is no Hauptmodul, but rather where the modular curve $X_0(N)$ has genus $1$. In those cases, the modular curve is an elliptic curve $E$ of conductor $N$ defined over $\Q$, which (over $\C$) is isomorphic to the torus $\C/\Lambda_E$ for a full lattice $\Lambda_E\subset\C$ called the period lattice of $E$. Denote by $\widehat\zeta(\Lambda_E;z)$ the associated completed Weierstrass zeta function (see \Cref{secWeierstrass} for the precise definition). For simplicity, we state the theorem just for the prime levels in question. The analogous statement for square-free composite levels is given in \Cref{thmdimcomp}.
\begin{theorem}\label{thmdim}
Let $V$ be a holomorphic, strongly rational vertex operator algebra of central charge $24$. Let $G=\langle g \rangle$ be a cyclic group of automorphisms of $V$ of order $p\in\{11,17,19\}$ such that $g$ is of type $p\{0\}$. Further let $E=X_0(p)$ be the $\Gamma_0(p)$-optimal elliptic curve of conductor $p$. Then with the assumptions and notations in \Cref{secVOA}, we have the following dimension formula:
\begin{multline*}
\dim V_1+\dim V_1^{\orb(g)}=(p+1)\dim V_1^G-(p-1)C_E\\
+C_E\sum_{i=1}^{p-1}\sum_{j=1}^{p-1} \sigma(p-j)\dim V(g^i)_{j/p},
\end{multline*}
where we set
$$C_E:=-\frac{3-\#E(\F_2)}{2}-\widehat\zeta\left(\Lambda_E;L(E,1)\right).$$
\end{theorem}
In particular, this dimension formula relates invariants of the underlying modular curve to the theory of VOAs. We can now exploit knowledge about arithmetic properties of these invariants to derive a simpler dimension formula in the following way.

To the best of the authors' knowledge the question of the rationality of the value $\widehat\zeta\left(\Lambda_E;L(E,1)\right)$ has not been investigated so far. It is known due to a classical result of Schneider \cite[Chapter II, \S 4, Satz 15]{Schneider} that the value of the uncompleted zeta function $\zeta\left(\Lambda_E;L(E,1)\right)$ is in fact transcendental. Since all the other quantities in the above dimension formula are clearly rational, we obtain the following immediate corollary.
\begin{corollary}\label{cordim}
Assume the notations as in \Cref{thmdim}. If we have
$$\sum_{i=1}^{p-1}\sum_{j=1}^{p-1} \sigma(p-j)\dim V(g^i)_{j/p}\neq p-1$$
for some VOA $V$ as in \Cref{thmdim}, then the value $\widehat\zeta\left(\Lambda_E;L(E,1)\right)$ is rational.
\end{corollary}
Computing the zeta values in \Cref{cordim} numerically, we find that within computational precision $\widehat\zeta\left(\Lambda_E;L(E,1)\right)=17/5,\ 2,\ 4/3$ for $p=11,\ 17,\ 19$, respectively.

In a very recent preprint \cite{EMS2}, M\"oller, and Scheithauer independently find a completely general dimension formula like the one in \Cref{thmdim} with no restriction on the order of the cyclic group $G$ using expansions of vector-valued Eisenstein series. In particular, their general result simplifies to the statement of the Theorem from \cite{Mollerthesis} quoted above for all primes $p$. 

In the proof of Theorem 4.12 of loc. cit., M\"oller and Scheithauer show that if the order of the automorphism group $G$ is any prime $p$ such that the genus of $X_0(p)$ is greater than zero, one obtains the lower bound 
$$\sum_{i=1}^{p-1}\sum_{j=1}^{p-1} \sigma(p-j)\dim V(g^i)_{j/p}\geq p-1.$$
As M\"oller and Scheithauer informed us, they have produced --- using both computer calculations and theoretical considerations based on work by Chenevier and Lannes \cite{CL19} on $p$-neighbours of Niemeier lattices --- explicit examples of suitable VOAs for which the above inequality is strict, wherefore according to \Cref{cordim} the values $\widehat\zeta\left(\Lambda_E;L(E,1)\right)$ are indeed rational. In fact, comparing to the (extended) version of M\"oller's result and using the examples M\"oller and Scheithauer have constructed, one finds the stronger statement that we have indeed the following identity for the constant $C_E$ from \Cref{thmdim},
\begin{gather}
C_E=-\frac{24}{p-1}.
\end{gather}

From the proof of \Cref{thmdim} we can infer the following dimension formula as well, which looks similar to that in \Cref{cordim}. The proof relies on the so-called Bruinier-Funke pairing (see \Cref{propBF}).
\begin{theorem}\label{thmdim2}
Assume the hypotheses and notation from \Cref{thmdim}, except that $p$ may now denote any prime number, and let $f(\tau)=\sum_{n=1}^\infty a(n)e^{2\pi in\tau}\in S_2(p)$ be a newform with Atkin-Lehner eigenvalue $\eps\in\{\pm 1\}$. Then we have 
$$\sum_{i=1}^{p-1}\sum_{j=1}^{p-1} a(p-j)\dim V(g^i)_{j/p}=-\eps p-a(p).$$
\end{theorem} 
Essentially, the formula in \Cref{thmdim2} also appears on \cite[p. 24]{EMS2}, but was proven using a different kind of pairing. We note that loosely speaking, one may interpret \Cref{thmdim} in view of \Cref{thmdim2} as the case where one replaces the newform $f$ by the weight $2$ Eisenstein series in $M_2(p)$.

\medskip

Our proof of \Cref{thmdim} relies on the following result which states that one can express any harmonic Maa\ss~form (in the given levels) essentially in terms of Weierstrass mock modular forms and Hecke operators (see \Cref{secWeierstrass,secOperators} for details). 
\begin{theorem}\label{thmspan}
Let $E$ denote the strong Weil curve of conductor 
$$N\in\{11,14,15,17,19,21\}.$$
Then any harmonic Maa{\ss} form of weight $0$ for $\Gamma_0(N)$ is is a linear combination of images of the completed Weierstrass mock modular form $\Zhat_E$ associated to the $\Gamma_0(N)$-optimal elliptic curve $E$ --- i.e. in the cases considered $E$ is a model for the modular curve $X_0(N)$ --- under the Hecke operators $T_m$ and 
Atkin-Lehner involutions
, or in other words:
$$H_0(N)\leq \span_\C\left\{\Zhat_E|W_Q|T_m|B_d \: :\: m\in\N_0,\ Q\mid N,\ d\mid N\right\},$$
where the operators $B_d$ are defined in \Cref{propAL}.
\end{theorem}
\begin{remark}
It is essential in \Cref{thmspan} that the elliptic curves under consideration are indeed models for the modular curve $X_0(N)$, where $N$ is the respective conductor. In particular, the genus of $X_0(N)$ must be equal to $1$. There are six further levels $N$ with this property, namely $N\in\{20,24,27,32,36,49\}$, but our proof does not work in these cases for reasons we explain in \Cref{secOperators,secProofsspan}.
\end{remark}
\begin{remark}\label{remgeneralize}
As our proof will show, the statement of \Cref{thmspan} remains valid for all square-free levels $N$ if one replaces $\Zhat_E$ by the Maa\ss-Poincar\'e series for $\Gamma_0(N)$ which has exactly one simple pole at the cusp $\infty$. In particular, one may immediately generalize \Cref{thmdim} to arbitrary primes $p$ and \Cref{thmdimcomp} to arbitrary square-free numbers $N$ in this fashion. This way, one may obtain rationality results for the constant terms of these series in analogy to Corollary 1.2. It is however not known to the authors whether these constant terms are directly related to special values of interesting functions like the Weierstrass zeta function.
\end{remark}
The rest of this paper is organized as follows. In \Cref{secPrelim} we recall some background material on orbifold constructions of VOAs, Weierstrass mock modular forms, and operators on modular forms. \Cref{secProofs} contains the proofs of \Cref{thmdim}, \Cref{thmdim2}, and \Cref{thmspan} as well as a more general version of \Cref{thmdim} and its proof.

\section*{Acknowledgements}
This research was initiated when both authors attended a workshop on moonshine at the Erwin-Schr\"odinger-Institut in Vienna in September 2018. The authors thank the institute for their hospitality. They also thank Scott Carnahan, John Duncan, Soon-Yi Kang, Sven M\"oller, Ken Ono, Nils Scheithauer, and David Zureick-Brown for helpful discussions. Furthermore, the authors are grateful to John Duncan, Sven M\"oller, and Nils Scheithauer, and the anonymous referee for their comments on earlier versions of this manuscript which helped to improve the exposition. 

\section{Preliminaries}\label{secPrelim}
\subsection{Orbifold VOAs}\label{secVOA}
The construction of the moonshine module $V^{\natural}$ \cite{FLM84, FLM83, FLM} has greatly motivated the study of vertex operator algebras (VOAs). The problem of orbifolding a conformal field theory with respect to an automorphism rose to prominence contemporaneously in physics \cite{DHVW2,DHVW1}. The construction of $V^{\natural}$ was subsequently interpreted as the first example of an orbifold model that is not equivalent to a lattice vertex operator algebra \cite{FLM}. For $G$ a group of automorphisms of $V$, the study of the fixed point sub-VOA $V^G$ and its representation theory is referred to as \emph{orbifold theory}.  We refer the reader to \cite{DRX15, EMS15, Mollerthesis} for details on cyclic orbifold theory for holomorphic VOAs and give a short summary below.

We first recall some basic definitions and properties of VOAs and their twisted modules. We refer the reader to \cite{FBZ} \cite{FLM} and \cite{LL12} for more details.

A VOA $V$ is a complex vector space equipped with two distinguished vectors $\bf{1}$ and $\omega$ called the vacuum element and the conformal vector, respectively. Further, for each vector $v\in V$ there is a map $Y(\cdot, z): V \to \text{End}(V)\llbracket z, z^{-1}\rrbracket$ assigning a formal power series $Y(v, z):=\sum_{n\in\mathbb{Z}} v(n) z^{-n-1}$ (which we call a vertex operator) to $v$. The tuple $(V,{\bf{1}},\omega, Y)$ must satisfy several axioms (see Section 8.10 of \cite{FLM}). In particular, the coefficients of the vertex operator attached to the conformal vector generate a copy of the Virasoro algebra of central charge $c$. In other words, if $Y(\omega, z):=\sum_{n\in\mathbb{Z}} L(n) z^{-n-2}$ then $[L(m),L(n)]=(m-n)L(m+n)+ \frac{1}{12}(m^3-m)\delta_{m+n,0}c ,$ and we refer to $c$ as the \textit{central charge} of $V$.
VOAs admit a $\mathbb{Z}$-grading (bounded from below) so that $V=\bigoplus_{n\in \mathbb{Z}} V_n$. This grading on $V$ comes from the eigenspaces of the $L(0)$ operator, by which we mean that $V_n:=\{ v\in V \mid L(0)v=nv \}$. The smallest $n$ for which $V_n\neq 0$ is called the \textit{conformal weight} of $V$ and is denoted $\rho(V)$. We say $V$ is of \textit{CFT-type} if $\rho(V)=0$ and $V_0=\mathbb{C}{\bf{1}}$.

A $V$-module is a vector space $M$ equipped with an operation $Y_M : V \to \text{End} M \llbracket z^{\pm1}\rrbracket$ which assigns to each $v\in V$ a formal power series $Y_M(v, z):=\sum_{n\in\mathbb{Z}} v^{M}(n) z^{-n-1}$ subject to several axioms (see section 5.1 of \cite{FBZ}). A module $M$ whose only submodules are $0$ and itself is called \textit{simple} or \textit{irreducible}. A VOA $V$ for which every admissible $V$-module decomposes into a direct sum of (ordinary) irreducibles is called \textit{rational} and we say that $V$ is \textit{holomorphic} if it is rational and has a unique irreducible module (which must necessarily be $V$ itself). Given a $V$-module $W$ with a grading, it is possible to define a $V$-module $W'$, that is (as a vector space) the graded dual space of $W$ (for a definition of the dual module we refer to Section $5.2$ of \cite{FHL93}). We say a vertex algebra $V$ is \textit{self-dual} if the module $V$ is isomorphic to its dual $V'$ (as a $V$-module). In \cite{Zhu}, Zhu introduced a finiteness condition on a VOA $V$, we say $V$ is \emph{$C_2$-cofinite} if $C_2(V ) := \text{span}\{v(2)w \mid v, w \in V \}$ has finite codimension in $V$. 
A VOA is called \textit{strongly rational} if it is rational, $C_2$-cofinite,  self-dual, and of CFT-type.

For $G$ a finite group of automorphisms of $V$ and $g\in G$, one can define a $g$-twisted module $V(g)$ of $V$ (see Section $3$ of \cite{DLM}). By \cite{DLM}, for $V$ a $C_2$-cofinite holomorphic VOA and $G=\langle g \rangle$ a cyclic group of automorphisms of $V$, $V$ posesses a unique simple $g^i$-twisted $V$-module, which we call $V(g^i)$, for each $i\in \mathbb{Z}/N\mathbb{Z}$ for $N$ the order of $g$. 
By Proposition $4.2.3$ of \cite{Mollerthesis} (see also \cite{DLM}) there is a representation 
$$\phi_i\colon G\to \text{Aut}_{\mathbb{C}}(V(g^i))$$
of $G$ on the vector space $V(g^i)$ such that $\phi_i(g)Y_{V(g^i)}(v,z)\phi_i^{-1}(g)= Y_{V(g^i)}$
for all $i\in \mathbb{Z}/N\mathbb{Z}$ and $v\in V$. This representation is unique up to an $N$-th root of unity. The eigenspace of $\phi_i(g)$ in $V(g^i)$ corresponding to the eigenvalue $e^{(2\pi i)j/N}$ is denoted by $W^{(i,j)}$ and as $\mathbb{C}[G]$-modules, we have that $V(g^i)=\bigoplus_jW^{(i,j)}$.

The fixed point sub-VOA $V^G=W^{(0,0)}$ of $V$ is defined to be the vectors in $V$ which are fixed pointwise under the action of $G$. 


 The main theorem of orbifold theory 
 \cite{CM16, Miy15, DM97} is that if $V$ is strongly rational and $G$ is a finite, solvable group of automorphisms of $V$, then the fixed-point VOA $V^G$ is strongly rational as well. 
 
For all $i,j\in\Z/N\Z$, the $W^{(i,j)}$ are irreducible $V^G$-modules \cite{Yam} and further, by the classification of irreducible modules in \cite{Mollerthesis}, there are exactly $n^2$ irreducible $V^G$-modules (namely, the $W^{(i,j)}$). We make the additional assumption that $g$ has type $N\{0\}$ (a certain condition on the conformal weights of the $g$-twisted modules, see Definition $4.7.4$ of \cite{Mollerthesis}), which gives us that the conformal weights obey $\rho (V (g)) \in (1/N)\mathbb{Z}$. This enables us to choose representations $\phi_i$ such that the conformal weights of  $W^{(i,j)}$ obey $\rho (W^{(i,j)}) \in (ij/N)\mathbb{Z}$.

For $V$ with central charge divisible by $24$, the characters of the irreducible $V^G$-modules $\text{ch}_{W^{(i,j)}}(\tau)= \text{tr}_{W^{(i,j)}} q^{L(0)-c/24}$ are holomorphic on the upper half-plane and modular of weight $0$ for $\Gamma_0(N)$ (\cite[Theorem 5.1]{EMS15}).

We also assume that $V^G$ satisfies the positivity assumption, which states that for a simple VOA $V$, the conformal weights of any irreducible $V$-module $W\neq V$ are positive and the conformal weight of $V$ is zero. 

If $V^G$ satisfies the positivity assumption, the orbifold VOA of $V$ with respect to $g$ is defined to be $$V^{\text{orb}(g)}:=\bigoplus_{i\in \mathbb{Z}/N\mathbb{Z}} W^{(i,0)}.$$
Note that if $V$ is strongly rational, then $V^{\text{orb}(g)}$ has the structure of a holomorphic, strongly rational VOA of the same central charge as $V$.
 
We refer to work by Zhu \cite{Zhu}, Dong--Li--Mason \cite{DLM}, and Dong--Lin--Ng \cite{DLNg} for details on the modular invariance of irreducible (twisted) modules for $C_2$-cofinite VOAs. We recall the following results from \cite{EMS15,MSV}.
\begin{proposition}\label{proptrans}
The characters $\ch_{W^{(i,j)}}(\tau)$ form a vector-valued modular form of weight $0$ for the Weil representation associated to the finite quadratic module $(\mathbb{Z}/N\mathbb{Z})\times (\mathbb{Z}/N\mathbb{Z})$ endowed with the quadratic form $q((i,j))=ij/N+\mathbb{Z}$. In particular, their transformation properties under the standard generators $S=\left(\begin{smallmatrix} 0 & -1 \\ 1 & 0\end{smallmatrix}\right)$ and $T=\left(\begin{smallmatrix} 1 & 1 \\ 0 & 1\end{smallmatrix}\right)$ of $\SLZ$ are given by 
$$\ch_{W^{(i,j)}}(S.\tau)=\frac{1}{n} \sum\limits_{k,\ell \in \mathbb{Z}/N\mathbb{Z}} e^{(2\pi i) (i\ell+jk)/N}\emph{ch}_{W^{(k,\ell)}}(\tau), $$
and 
$$\ch_{W^{(i,j)}}(T.\tau)= e^{(2\pi i) ij/N}\emph{ch}_{W^{(i,j)}}(\tau).$$
\end{proposition}
From equation (7) of \cite{MSV}, we have the following transformation property.
\begin{proposition}\label{proptrans2}
The character $\emph{ch}_{V^G}(\tau)$ is a modular function for $\Gamma_0(N)$ and moreover, for a matrix $\gamma= \left(\begin{smallmatrix} a & b \\ c & d \end{smallmatrix} \right) \in \SLZ$, sending $\infty$ to the cusp $\fraka=a/c$ with $c\mid N$ and $\gcd(a,c)=1$, we have\footnote{Note that in \cite[Equation (7)]{MSV} there is an erroneous minus sign in the exponential.}
$$ \ch_{W^{(0,0)}}(\gamma.\tau)=\frac{c}{N} \sum\limits_{i, j\in \mathbb{Z}/(N/c)\mathbb{Z}} e^{(2\pi i) dcij/N}\emph{ch}_{W^{(ci,cj)}}(\tau).$$
 \end{proposition}
For a cusp $\fraka$ of $\Gamma_0(N)$, van Ekeren, M\"oller, and Scheithauer \cite{MSV} define the function
\begin{gather}\label{eqFfraka} 
F_\fraka(\tau):=\sum\limits_{\substack{\gamma\in \Gamma_0(N)\setminus \SLZ \\ \gamma.\infty=\fraka}} \text{ch}_{W^{(0,0)}}(\gamma.\tau). 
\end{gather}
In \cite[Proposition 3.6]{MSV}, they give the following general identity for this function $F_\fraka$ which is essential in establishing the dimension formulas.
\begin{proposition}\label{propEMS}
The function $F_\fraka$ defined in \eqref{eqFfraka} satisfies the identity
$$\sum_{\fraka} F_\fraka(\tau)=\sum\limits_{d\mid N} \dfrac{\varphi(\gcd(d,N/d))}{\gcd(d,N/d)}\ch_{V^{\orb(g^d)}}(\tau),$$
where the sum over $\fraka$ runs over a set of representatives of cusps of $\Gamma_0(N)$ and $\varphi(n):=\#(\Z/n\Z)^*$ denotes Euler's totient function.
\end{proposition}

\subsection{Mock modular forms and harmonic Maa\ss~forms}
In this section, we briefly recall some basic definitions and facts about mock modular forms and harmonic Maa\ss~forms. For more detailed information as well as references to original works, the reader may consult for example the book \cite{BOOK}.

A \emph{harmonic Maa\ss~form} of weight $k\in\Z$ for the group $\Gamma_0(N)$ is a smooth function $f:\HH\to\C$ satisfying the following three conditions:
\begin{enumerate}
\item $f$ is invariant under the weight $k$ slash operator,
$$f|_k\abcd a b c d :=(c\tau+d)^{-k} f\left(\frac{a\tau+b}{c\tau+d}\right)=f(\tau)\quad \text{for all $\tau\in\HH$ and $\abcd a b c d \in\Gamma_0(N)$.}$$
\item $f$ is annihilated by the weight $k$ hyperbolic Laplacian ($\tau=x+iy$),
$$\Delta_k f := \left[-y^2\left(\frac{\partial^2}{\partial x^2}+\frac{\partial^2}{\partial y^2}\right)+iky\left(\frac{\partial}{\partial x}+i\frac{\partial}{\partial y}\right)\right]f\equiv 0.$$
\item $f$ has at most linear exponential growth at the cusps, i.e. there exists a polynomial $H\in\C[X]$ such that $f-H(q^{-1})$ has exponential decay towards infinity and analogous conditions hold at all other cusps.
\end{enumerate}
The space of these forms is denoted by $H_k(N)$. The subspaces $S_k(N)\subseteq M_k(N)\subseteq M_k^!(N)\subseteq H_k(N)$ denote the spaces of cusp forms, modular forms, and weakly holomorphic modular forms. It is sometimes convenient to relax the conditions to allow poles in the upper half-plane, in which case we speak of \emph{polar} harmonic Maa\ss~forms.

These functions naturally split into a holomorphic and a non-holomorphic part \cite[Lemma 4.3]{BOOK}, $f=f^++f^-$. The holomorphic part of a harmonic Maa\ss~form is called a \emph{mock modular form}. If $f$ is a polar harmonic Maa\ss~form we call $f^+$ a polar mock modular form.  Vice versa, given a mock modular form $f$, we refer to the harmonic Maa\ss~form $\widehat f$ having it as its holomorphic part as the \emph{(modular) completion} of $f$. 

The non-holomorphic part of a harmonic Maa\ss~form is related to a cusp form called the \emph{shadow} of the corresponding mock modular form \cite[Theorem 5.10]{BOOK}.
\begin{proposition}\label{propxi}
The operator $\xi_k=2iy^k\overline{\frac{\partial}{\partial\overline\tau}}$ defines a surjective $\C$-antilinear map
$$H_k(N)\twoheadrightarrow S_{2-k}(N)$$
with kernel $M_k^!(N)$.
\end{proposition}
An important tool obtained from the $\xi$-operator is the so-called \emph{Bruinier-Funke pairing}, defined by
$$\{\cdot,\cdot\}: M_k(N)\times H_{2-k}(N)\to\C,\ \{g,f\}:=\langle g,\xi_{2-k}f\rangle,$$
where for $g_1,g_2\in M_k(N)$ such that $g_1g_2$ is a cusp form we define
$$\langle g_1,g_2\rangle :=\frac{1}{[\SLZ:\Gamma_0(N)]}\int_{\Gamma_0(N)\setminus\HH} g_1(\tau)\overline{g_2(\tau)}y^k\frac{dxdy}{y^2}$$
as the classical Petersson scalar product. With this we have the following important result (see \cite[Proposition 5.10]{BOOK}), which follows essentially from an application of Stokes's Theorem.
\begin{proposition}\label{propBF}
Let $g\in M_k(N)$ and $f\in H_{2-k}(N)$. For a cusp $\fraka$ of $\Gamma_0(N)$ of width $h$, fix $\gamma\in\SLZ$ with $\gamma.(i\infty)=\fraka$ and consider the Fourier expansions 
$$(g|\gamma)(\tau)=\sum_{n=0}^\infty a_\fraka(n)q^{n/h}\quad\text{and}\quad (f|\gamma)^+(\tau)=\sum_{m\gg-\infty} b_\fraka(n)q^{n/h}.$$
Then we have
$$\{g,f\}=\sum_\fraka\sum_{n\leq 0} a_\fraka(-n)b_\fraka(n).$$
\end{proposition}
An easy and well-known consequence of this is the following corollary.
\begin{corollary}\label{corBF}
A harmonic Maa{\ss} form in $H_{2-k}(N)$ with no pole at any cusp is a holomorphic modular form.
\end{corollary}
\subsection{Operators on modular forms}\label{secOperators}
Hecke operators are certainly the most important operators on modular forms. We first review their definition and some basic properties. For this, consider for any $N,m\in\N$ the set
$$\calM_m(N)=\left\{M=\abcd a b c d\in\Z^{2\times 2}\: :\: \det M=m,\ N\mid c,\ \gcd(a,N)=1\right\}.$$
The group $\Gamma_0(N)$ acts on $\calM_m(N)$ by left-multiplication and we let $\beta_1,...,\beta_s$ denote a set of coset representatives of this action. For any function $f:\HH\to \C$ transforming like a modular form of weight $k\in\Z$ under $\Gamma_0(N)$ we then define the $m$-th Hecke operator acting on $f$ by
\begin{gather}\label{eqHeckedef}
f|T_m^{(N)}=f|T_m=m^{k/2-1}\sum_{\beta\in\Gamma_0(N)\setminus\calM_m(N)} f|_k\beta,
\end{gather}
where we extend the action of the weight $k$ slash operator to matrices with positive discriminant in the usual way by
$$(f|_k\gamma)(\tau)=(\det \gamma)^{k/2}(c\tau+d)^{-k}f\left(\frac{a\tau +b}{c\tau+d}\right).$$
We usually omit the indication of the level of the Hecke operator if it is clear from context or not relevant for the action.

These operators form a commutative algebra and they are multiplicative, i.e. one has $T_mT_n=T_{mn}$ for any coprime $m,n$. Their action on Fourier expansions is particularly easy to describe when $m=p$ is prime. Then we have
$$f|T_p=\begin{cases} f|U_p+p^{k-1}f|B_p & \text{if }p\nmid N \\ f|U_p & \text{if }p\mid N\end{cases}$$
where for $f(\tau)=\sum_{n\in\Z}a_f(n,y)q^n$ we set
$$(f|B_m)(\tau)=f(m\tau)=\sum_{n\in\Z} a_f(n,y)q^{mn}\qquad\text{and}\qquad (f|U_m)(\tau)=\sum_{n\in\Z} a_f(mn,y/m)q^n.$$
We record the following easy corollary (\cite[Corollary 13.3.10]{CS}) of the Multiplicity 1-Theorem (see for instance \cite[Theorem 13.3.9]{CS}).
\begin{lemma}\label{lemnew}
Let $H$ be an operator on the space $S_k(N)$ of weight $k$ cusp forms for $\Gamma_0(N)$ which commutes with all Hecke operators $T_p$ for $p\nmid N$. Then any newform is an eigenfunction of $H$.
\end{lemma}
Another important set of operators is given by the \emph{Atkin-Lehner involutions}. For any exact divisor $Q$ of $N$, i.e. $Q\mid N$ and $\gcd(Q,N/Q)=1$, we define the Atkin-Lehner operator via the matrix 
\begin{gather}\label{defAL}
W_Q=\frac{1}{\sqrt{Q}}\begin{pmatrix}
Qx & y \\ Nz & Qt
\end{pmatrix},
\end{gather}
where $x,y,z,t\in\Z$ are chosen so that $\det W_Q=1$. In the following proposition, we collect several well-known properties of these operators which will become important in the proof of \Cref{thmspan}. These can be found for instance in \cite[Lemma 6.6.4, Proposition 13.2.6]{CS}.
\begin{proposition}\label{propAL}
Let $m\in\N$ and $Q,Q'$ exact divisors of $N$ and let $f:\HH\to\C$ be a function transforming like a modular form of weight $k\in 2\Z$ for $\Gamma_0(N)$. Then the following are true.
\begin{enumerate}[(i)]
\item As matrices, we have $W_Q=B_Q^{-1}\gamma=\gamma'B_Q$ for $\gamma,\gamma'\in\Gamma_0(N/Q)$ and where we set $B_m:=\frac{1}{\sqrt{m}}\left(\begin{smallmatrix} m & 0 \\ 0 & 1\end{smallmatrix}\right)$.
\item $W_Q$ normalizes $\Gamma_0(N)$.
\item We have $W_Q^2\in\Gamma_0(N)$ and $f|W_Q|W_{Q'}=f|W_{Q'}|W_Q=f|W_{QQ'}$
\item For $\gcd(m,Q)=1$ we have $f|W_Q|T_m=f|T_m|W_Q$.
\item For $\gcd(m,Q)=1$ we have $f|W_Q|B_m=f|B_m|W_Q$.
\item If $Q=p$ is prime, then $f|U_p+p^{k/2-1}f|W_p$ transforms like a modular form for $\Gamma_0(N/p)$.
\end{enumerate}
\end{proposition}
\subsection{Weierstrass mock modular forms}\label{secWeierstrass}
In this section, we briefly recall the construction of Weierstrass mock modular forms. The idea for this construction is due to Guerzhoy \cite{Guerzhoy1,Guerzhoy2} and was developed further by Alfes, Griffin, Ono, and Rolen \cite{AGOR}.

Let $E$ be an elliptic curve defined over $\Q$ of conductor $N$ defined by the Weierstrass equation
$$E:\:y^2=4x^3-g_2x-g_3.$$
As mentioned in the introduction, this curve (considered over $\C$) is isomorphic to a flat torus $C/\Lambda_E$, where $\Lambda_E\subset \C$ is a $2$-dimensional $\Z$-lattice. This isomorphism is given by
$$\C/\Lambda_E\to E,\: z+\Lambda_E\mapsto \begin{cases} (\wp(\Lambda_E;z),\wp'(\Lambda_E;z)) & \text{if }z\notin\Lambda_E \\ \calO & \text{otherwise},\end{cases}$$
where
$$\wp(\Lambda_E;z)= \frac{1}{z^2}+\sum_{\omega\in\Lambda_E\setminus\{0\}} \left(\frac{1}{(z-\omega)^2}-\frac{1}{\omega^2}\right)$$
denotes the Weierstrass $\wp$-function and $\calO\in E$ denotes the point at infinity. Recall that $\wp(\Lambda_E;z+\omega)=\wp(\Lambda_E;z)$ for all $\omega\in\Lambda_E$ and in fact the field of all \emph{elliptic functions}, i.e. meromorphic functions with this exact periodicity property, is given by $\C(\wp)[\wp']$, where $\wp$ satisfies the differential equation
$$(\wp')^2=4\wp^3-g_2-g_3.$$
The $\wp$-function has poles of order $2$ with residue $0$ at all lattice points by construction. Its Laurent expansion around $0$ is given by
$$\wp(\Lambda_E;z)=\frac{1}{z^2}+\sum_{n=2}^\infty (2n-1)G_{2n}(\Lambda_E)z^{2n-2},$$
where for integers $k>2$, $G_{2n}(\Lambda_E)=\sum_{\omega\in\Lambda_E\setminus\{0\}} \omega^{-k}$ denotes the weight $k$ \emph{Eisenstein series} of $\Lambda_E$, which is of course $0$ if $k$ is odd. The negative antiderivative of the Weierstrass $\wp$-function, called the Weierstrass $\zeta$-function, therefore has simple poles at all lattice points and nowhere else and is given by
$$\zeta(\Lambda_E;z)=\frac 1z +\sum_{\omega\in\Lambda_E\setminus\{0\}}\left(\frac{1}{z-\omega}+\frac 1\omega+\frac{z}{\omega^2}\right)=\frac 1z -\sum_{n=2}^{\infty} G_{2n}(\Lambda_E)z^{2n-1}.$$
However, by Liouville's famous theorems on elliptic functions, there cannot be an elliptic function with simple poles only at lattice points and nowhere else, so $\zeta(\Lambda_E;z)$ is not quite an elliptic function. It was first observed by Eisenstein (in a special case) that there is a canonical way to complete the Weierstrass $\zeta$-function to a function which has the periodicity behaviour of an elliptic function at the expense of no longer being holomorphic. In order to define Eisenstein's completed Weierstrass $\zeta$-function let
$$G_2^*(\Lambda_E)=\lim_{s\to 0} \sum_{\omega\in\Lambda_E\setminus\{0\}} \omega^{-2}|\omega|^{-2s}$$
denote the completed\footnote{The sum defining the Eisenstein series is no longer absolutely convergent for $k=2$. The modification here is sometimes called Hecke's trick \cite{Hecke}.} Eisenstein series of weight $2$. By the famous modularity theorem, there is a newform $f_E\in S_2(N)$ with integer Fourier coefficients associated to $E$ such that the $L$-functions of $E$ and $f_E$ agree, which by Eichler-Shimura theory yields a polynomial map
$$\phi_E:X_0(N)\to \C/\Lambda_E,$$
the \emph{modular parametrization} of $E$. Then the non-holomorphic function
$$\widehat\zeta(\Lambda_E;z)=\zeta(\Lambda_E;z)-G_2^*(\Lambda_E)z-\frac{\deg \phi_E}{4 \pi \Vert f_E\Vert^2}\overline z,$$
where $\Vert \cdot\Vert$ denotes the Petersson norm, satisfies $\widehat\zeta(\Lambda_E;z+\omega)=\widehat\zeta(\Lambda_E;z)$ for all $z\in\C\setminus\Lambda_E$ and $\omega\in\Lambda_E$.

The newform $f_E$ has a Fourier expansion $f_E(\tau)=\sum_{n=1}^\infty a_E(n)q^n$ with $q=e^{2\pi i\tau}$. Denoting by 
$$\Eich_E(\tau)=-2\pi i\int_\tau^{\infty} f_E(t)dt =\sum_{n=1}^\infty \frac{a_E(n)}{n}q^n$$
the \emph{Eichler integral} of $f_E$, one finds the following result \cite[Theorem 1.1]{AGOR}.
\begin{theorem}
The function 
$$\Zed_E(\tau)=\zeta(\Lambda_E;\Eich_E(\tau))-G_2^*(\Lambda_E)\Eich_E(\tau),$$
called the \emph{Weierstrass mock modular form} is a polar mock modular form of weight $0$ for the group $\Gamma_0(N)$. To be more precise, there exists a meromorphic modular function $M_E$ for $\Gamma_0(N)$ such that the function
$$\Zhat_E(\tau)=\widehat\zeta(\Lambda_E;\Eich_E(\tau))-M_E(\tau)$$
is a harmonic Maa\ss~form of weight $0$ for $\Gamma_0(N)$.
\end{theorem}
It is immediately clear from the definition that the function $\Zed_E$ has poles precisely where the value of the Eichler integral $\Eich_E(\tau)$ lies in the period lattice $\Lambda_E$. It is an open problem to classify those points $\tau$ in the complex upper half-plane $\HH$ where this occurs, but the following lemma, whose proof can be found for example in \cite{AliMani}, allows us to rule out poles in the situation where $E$ and the modular curve $X_0(N)$ are actually isomorphic, so where the degree of $\phi_E$ is $1$.
\begin{lemma}
Let $E$ be the strong Weil curve of conductor $N$ such that $X_0(N)$ has genus $1$, i.e. $N\in\{11,14,15,17,19,20,21,24,27,32,36,49\}$. Then the Weierstrass mock modular form $\Zed_E$ has no poles in $\HH$. 
\end{lemma}
For the purpose of this paper, it is important to consider the behaviour of the (completed) Weierstrass mock modular form at other cusps than infinity. For this, we need the following slight generalization of \cite[Theorem 1.2]{AGOR}.
\begin{proposition}\label{propWZnormalizer}
Let $\nu\in\Norm(\Gamma_0(N))$, the normalizer of $\Gamma_0(N)$ in $\SL_2(\R)$, which commutes with all Hecke operators $T_p$ with prime $p\nmid N$. Then we have
$$\left(\Zhat_E|_0\nu\right)(\tau)=\widehat\zeta(\Lambda_E;\lambda_\nu(\Eich_E(\tau)-\Omega_{\nu^{-1}}(f_E)))$$
where $\Omega_\nu(f_E)=-2\pi i\int_{\nu^{-1}\infty}^\infty f_E(t)dt$ and $\lambda_\nu$ is the eigenvalue of $f_E$ under $\nu$ (see \Cref{lemnew}). In particular for $-\lambda_\nu\Omega_\nu(f_E)\notin\Lambda_E$ we find the asymptotic
$$(\Zed_E|_0\nu)(iy)\sim \widehat\zeta(-\lambda_\nu\Omega_\nu(f_E))+\exp(-\alpha y)\qquad\text{as } y\to\infty,$$
for some $\alpha>0$.
\end{proposition}
\begin{proof}
In \cite[Theorem 1.2]{AGOR}, the result is stated for $\nu$ an Atkin-Lehner involution. The exact same proof goes through, only applying \Cref{lemnew} in the substitutions. In fact, the computation goes through even for any elliptic curve $E/\Q$ and matrix $\sigma\in\SL_2(\R)$ such that $\sigma\Gamma_0(N)\sigma^{-1}\cap\Gamma_0(N)$ has finite index in $\Gamma_0(N)$: One finds
\begin{gather}\label{eqZEcusp}
\begin{aligned}
(\Zhat_E|_0\sigma)(\tau)&=\widehat\zeta\left(\Lambda_E;-2\pi i\int_{\sigma.\tau}^{i\infty} f_E(z)dz\right)\\
&=\widehat\zeta\left(\Lambda_E;-2\pi i\int_{\tau}^{i\infty} (f_E|_2\sigma)(z)dz+2\pi i\int_{\sigma^{-1}.(i\infty)}^{i\infty} (f_E|_2\sigma)(z)dz\right)\\
&=\widehat\zeta\left(\Lambda_E;-2\pi i\int_{\tau}^{i\infty} (f_E|_2\sigma)(z)dz+\Omega_{\sigma^{-1}}(f_E)\right).
\end{aligned}
\end{gather}
The claim then follows immediately using \Cref{lemnew}.
\end{proof}
\begin{remark*}
In the case of interest to us, the space $S_2(N)$ is one-dimensional. Since $\Norm(\Gamma_0(N))$ acts on the space of cusp forms, a cusp form in those levels must be an eigenfunction under any element of the normalizer, so the proof goes through here without the appeal to the Multiplicity-one theorem and \Cref{lemnew}.
\end{remark*}
\begin{corollary}\label{corpoincare}
For all levels $N\in\{11,14,15,17,19,20,21,24,27,32,36\}$, the completed Weierstrass mock modular form $\Zhat_E$ for the strong Weil curve $E$ of conductor $N$ has a simple pole at $\infty$ and is constant at all other cusps.
\end{corollary}
\begin{proof}
For the given $N$, the normalizer $\Norm(\Gamma_0(N))$ acts transitively on the cusps of $\Gamma_0(N)$. Note that for the square-free levels, it is well-known that the Atkin-Lehner operators already act transitively on the cusps. By computing the relevant periods $\Omega_\nu(f_E)$ explicitly\footnote{The authors used the \texttt{mfsymboleval} command in \textsc{Pari/Gp} \cite{PARI} for this.}, we see that none of them are in $\Lambda_E$ and the claim follows. 
\end{proof}
\begin{remark*}
For the remaining level $49$, the normalizer does not act transitively on cusps. However, one can use the fact that 
$$\frac{1}{2\pi i}\frac{\partial}{\partial\tau}\Zhat_E(\tau)=\frac{1}{g_{49}(\tau)}\left(\frac{1}{2400}E_4(7\tau)-\frac{2401}{2400}E_4(49\tau)+G_{49}(\tau)\right),$$
where $g_{49}(\tau)=q + q^2 - q^4 - 3q^8 - 3q^9 + O(q^{11})$ denotes the unique newform in $S_2(49)$ and $G_{49}(\tau)=-q + q^3 - q^4 - q^5 - q^6 + 49/10q^7 + 5q^8 + q^9 - 6q^{10} + 7q^{11} + O(q^{12})\in S_4(49)$, is a weakly holomorphic modular form of weight $2$. The fact that the derivative of $\Zhat_E$ is a weakly holomorphic modular form is a general consequence of Bol's identity and for the identification one notices that by \eqref{eqZEcusp}, the function $\Zhat_E$ and therefore its derivative can have at most a simple pole at any cusp, so that $\frac{1}{2\pi i}\frac{\partial}{\partial\tau}\Zhat_E\cdot g_{49}$ is a holomorphic modular form of level $49$. Due to this identity, it can be checked that its only pole is at infinity (in fact, it vanishes at all other cusps), wherefore, since differentiation commutes with the action of $\SL_2(\R)$ in weight $0$ and doesn't introduce or add any poles, \Cref{corpoincare} is also true for $N=49$. The same argument would of course also work in the cases covered by \Cref{corpoincare}.
\end{remark*}
\section{Proofs}\label{secProofs}
\subsection{Proof of \Cref{thmspan}}\label{secProofsspan}
Before proceeding to the proof of \Cref{thmspan}, we require a general result on the action of Hecke operators on Poincar\'e series. Recall that one may formally define a Poincar\'e series of weight $k$ for the group $\Gamma\leq \SLZ$ by averaging a suitably periodic seed function $\varphi:\HH\to\C$ over the coset representatives $\Gamma_\infty\setminus\Gamma$ where $\Gamma_\infty=\Stab_\Gamma(\infty)$ denotes the stabilizer of the cusp $\infty$, i.e.
$$\mathcal{P}(\Gamma,k,\varphi):=\sum_{\gamma\in \Gamma_\infty\setminus\Gamma} \varphi|_k\gamma.$$
If the defining series converges absolutely, this defines a function which transforms like a modular form of weight $k$ under $\Gamma$ (see for instance \cite[Lemma 8.2]{Ono} and the references therein). Restricting to the case of $\Gamma=\Gamma_0(N)$, one can compute their Fourier expansions fairly explicitly, especially in the most important cases for our purposes, where the seed function is either the exponential function $\varphi(\tau)=\exp(2\pi im\tau)$, $m\in\Z$, or a modified version of the Whittaker function (see \cite[Section 8]{Ono} for details) yielding harmonic Maa{\ss} forms. In those cases, the Fourier coefficients of the $m$th Poincar\'e take the general form
\begin{gather}\label{eqPoincarecoeff}
a_m^{(N,k)}(n)=C_k(n/m)^{(k-1)/2}\sum_{c=1}^\infty \frac{K(m,n,Nc)}{Nc}\calJ_k\left(\frac{\sqrt{mn}}{Nc}\right),
\end{gather}
where $C_k$ is some constant depending on the weight $k$, $\calJ_k$ is a suitable test function\footnote{In the cases considered, it is essentially a Bessel function.} so that the sum converges absolutely, and $K(m,n,c)$ denotes the Kloosterman sum
\begin{gather}\label{eqKloosterman}
K(m,n,c)=\sum_{d\,(c)^*}\exp\left(2\pi i\frac{m\overline d+nd}{c}\right)
\end{gather}
where the sum runs over all $d$ modulo $c$ with $\gcd(c,d)=1$ and $d\overline d\equiv 1\pmod c$. Kloosterman sums satisfy the so-called \emph{Selberg identity},
\begin{gather}\label{eqSelberg}
K(m,n,c)=\sum_{d\mid\gcd(m,n,c)} dK(1,mn/d^2,c/d).
\end{gather}
This identity was first noted without proof by Selberg \cite{Selberg}. The first published proof was found by Kuznetsov \cite{Kuznetsov} using his famous summation formula, and an elementary proof was found by Matthes \cite{Matthes}.
Using this, we can show the following general result.
\begin{proposition}\label{lemHeckePoincare}
Let $k\leq 0$, and $N,\nu\in\N$. Further denote the $\nu$th Maa{\ss}-Poincar\'e series of weight $k$ and level $N$ normalized so that its principal part at $\infty$ is given by $q^{-\nu}+O(1)$ by $P^{(N,k)}_\nu(\tau)$. Then we have that
\begin{gather}\label{eqHeckePoincare}
P^{(N,k)}_\nu=\sum_{d\mid\gcd(N,\nu)} (\nu/d)^{1-k}P^{(N/d,k)}_1|T_{\nu/d}^{(N/d)}|B_d.
\end{gather}
\end{proposition}
\begin{proof}
The action of Hecke operators in level $N$ may be solely defined on the space of one-periodic holomorphic functions together with a weight $k$ slash action of $\SL_2(\R)$ via their actions on Fourier expansions: For a function $f(\tau)=\sum_{n\in\Z} a(n)q^n$, we have (see for instance \cite[Proposition 10.2.5]{CS}) $(f|T_m^{(N)})=\sum_{n\in\Z} b(n)q^n$ where
\begin{gather}\label{eqHeckeFourier}
b(n)=\sum_{\substack{d\mid\gcd(m,n) \\ \gcd(d,N)=1}} d^{k-1}a(mn/d^2).
\end{gather}
Since a harmonic Maa{\ss} form is uniquely determined by its holomorphic part, we restrict our attention to the holomorphic part of the right-hand side of \eqref{eqHeckePoincare}. Using \eqref{eqHeckeFourier}, this is given by
$$\sum_{d\mid\gcd(N,\nu)} (\nu/d)^{1-k}\sum_{n\in\Z} \left[\sum_{\substack{t\mid\gcd(\nu/d,n) \\ \gcd(N/d,t)=1}} t^{k-1}a^{(N/d,k)}_1\left(\frac{n\nu/d}{t^2}\right)\right]q^{dn}.$$
The principal part of the holomorphic part is easily seen to equal $q^{-\nu}$ as desired, since the only term then surviving in the sum is the one with $d=\gcd(N,\nu)$ and $t=\nu/d$. The $n$th coefficient for $n>0$ is given by
$$\nu^{1-k}\sum_{d\mid\gcd(N,\nu,n)} \left[\sum_{\substack{t\mid\gcd(\nu/d,n/d) \\ \gcd(N/d,t)=1}} (dt)^{k-1}a^{(N/d,k)}_1\left(\frac{n\nu}{(dt)^2}\right)\right].$$
Plugging in the definition of the coefficient $a^{(N/d,k)}_1\left(\frac{n\nu}{(dt)^2}\right)$ then yields
\begin{gather*}
\begin{aligned}
 &C_k\nu^{1-k}\sum_{d\mid\gcd(N,\nu,n)}\sum_{\substack{t\mid\gcd(\nu/d,n/d) \\ \gcd(N/d,t)=1}} (\nu n)^{(k-1)/2}\sum_{c=1}^\infty \frac{K\left(1,\frac{\nu n}{(dt)^2},\frac Nd c\right)}{\frac Nd c}\calJ_k\left(\frac{\sqrt{\frac{\nu n}{(dt)^2}}}{\frac Nd c}\right)\\
=& C_k(n/\nu)^{(k-1)/2}\sum_{d\mid\gcd(N,\nu,n)}\sum_{\substack{t\mid\gcd(\nu/d,n/d) \\ \gcd(N/d,t)=1}} \sum_{c=1}^\infty (dt)\frac{K\left(1,\frac{\nu n}{(dt)^2},\frac {N(ct)}{dt}\right)}{N(ct)}\calJ_k\left(\frac{\sqrt{\nu n}}{N (ct)}\right).
 \end{aligned}
 \end{gather*}
We may and do assume without loss of generality that the defining series for the coefficient $a_1^{N,k}$ is absolutely convergent, since in the cases of interest to us, this can be achieved essentially by a Hecke-trick-like argument and analytic continuation (see for instance \cite[Chapter 1]{BruinierHabil} for details). Noticing further that every common divisor of $Nc,\nu,n$ can be uniquely factored into a common divisor $d$ of $N,\nu,n$ and a common divisor $t$ of $c,\nu/d,n/d$ which is coprime to $N/d$, we may rearrange the above sum to obtain the expression
$$C_k(n/\nu)^{(k-1)/2}\sum_{c=1}^\infty  \left[\sum_{d\mid\gcd(Nc,\nu,n)}dK\left(1,\frac{\nu n}{d^2},\frac {Nc}{d}\right)\right]\frac{\calJ_k\left(\frac{\sqrt{\nu n}}{N c}\right)}{Nc}.$$
By the Selberg identity \eqref{eqSelberg} we see by comparing to \eqref{eqPoincarecoeff} that this is exactly the coefficient $a_\nu^{(N,k)}(n)$. The constant terms may be compared through a similar but easier argument, which we refrain from carrying out here, therefore completing the proof.
\end{proof}
For levels $N\in\{1,...,10,12,13,16,18,25\}$, where the modular surface $X_0(N)$ has genus $0$, and weight $k=0$, the above result has been shown using ad hoc methods in \cite[Lemma 2.11]{BeneishLarson}, the result for general levels in weight $0$ is stated for $\gcd(N,\nu)=1$ and (incorrectly) for $\nu\mid N$ in \cite[Theorem 1.1]{BKLOR} (see also \cite{BKLOR2}). In \cite[Theorem 1.1]{Kangetal}, Jeon, Kang, and Kim give a slightly different proof of \Cref{lemHeckePoincare} for the case of weight $0$, which is not an important restriction, and use it to prove congruences for the Fourier coefficients of modular functions for genus $1$ levels (see loc. cit., Theorem 1.6). The analogous statement for cuspidal Poincar\'e series may also be obtained from the Petersson coefficient formula together with the fact that Hecke operators are Hermitian with respect to the Petersson inner product on the space of cusp forms, see \cite[Proposition 10.3.19]{CS} for the case $N=1$, but which is easily generalized to higher levels. 

\begin{proof}[Proof of \Cref{thmspan}]
According to \Cref{corpoincare}, the Weierstrass mock modular form is (up to a possible additive constant which is of no importance here) the first Maa{\ss}-Poincar\'e series in the respective level with its only pole at the cusp $\infty$. \Cref{lemHeckePoincare} shows how to obtain harmonic Maa{\ss} forms with arbitrarily high pole orders at infinity by the application of Hecke operators, provided that we are able to express the lower-level Poincar\'e series needed in \eqref{eqHeckePoincare} in terms of functions of the form $\Zhat_E|W_Q|T_m|B_d$, as required in the theorem. 

We therefore need to show first that we can generate the Poincar\'e series $P_1^{(d,0)}$ for all divisors $d\mid N$ for the relevant levels $N\in\{11,14,15,17,19,21\}$. By \Cref{propAL} $(vi)$, we know that $pf|U_p+f|W_p$ transforms under $\Gamma_0(N/p)$ for any function $f$ transforming like a modular function for $\Gamma_0(N)$. Applying this to the function $P_1^{(N,0)}|W_p$ shows that $P_1^{N,0}+p  P_1^{N,0}|W_p|U_p$ is a harmonic Maa{\ss} form on $\Gamma_0(N/p)$. Viewed as a harmonic Maa{\ss} form for $\Gamma_0(N)$, this function has a principal part $q^{-1}+O(1)$ at $\infty$ and by applying the Atkin-Lehner operator $W_p=\gamma\cdot \left(\begin{smallmatrix}\sqrt p & 0 \\ 0 & 1/\sqrt{p}\end{smallmatrix}\right)$ (cf. \Cref{propAL} $(i)$) to it once more, we see that it has principal part ${(q^{1/p})}^p+O(1)$ at the cusp $N/p$ of $\Gamma_0(N)$. The same is true for the Poincar\'e series $P_1^{N/p,0}$ viewed as a harmonic Maa{\ss} form for $\Gamma_0(N)$, so the two functions can only differ by a constant by \Cref{corBF}. Since $N$ is square-free by assumption, we can repeat this process to obtain all Poincar\'e series $P_1^{d,0}$ in this fashion, keeping in mind the commutation rules for Hecke and Atkin-Lehner operators in \Cref{propAL}.

\medskip 

Since the Atkin-Lehner operators act transitively on the cusps of $\Gamma_0(N)$, we have shown now that any harmonic Maa{\ss} form of level $N$ for the given $N$ can be written as a linear combination of functions of the form $P_1^{(N,0)}|W_Q|T_m|B_d|W_{Q'}$ for $Q,Q'\mid N$, $m\in\N_0$, where we set $f|T_0:=1$ for any function for convenience, and $d\mid N$. In order to complete the proof, we need to show that the application of $W_{Q'}$ may be avoided. The harmonic Maa{\ss} form in $H_0(N)$ which has a pole of order $m$ at the cusp $W_p\infty$ for a prime $p\mid N$ and nowhere else is given by
$$P_m^{(N,0)}|W_p=\sum_{d\mid\gcd(N,m)} (m/d)P^{(N/d,0)}_1|T_{m/d}^{(N/d)}|B_d|W_p$$
by \Cref{lemHeckePoincare}. For a common divisor $d$ of $m$ and $N$ not divisible by $p$ the operators  $B_d$ and $W_p$ commute (see e.g. \cite[Proposition 13.2.6 (d)]{CS}). Furthermore, we can write $T_{m/d}^{(N/d)}=U_{p^\ell}T_{m'/d}^{(N/d)}$ with $m=p^\ell m'$ and $p\nmid m'$. The Atkin-Lehner operator $W_p$ commutes with the Hecke operator $T_{m'/d}^{(N/d)}$ and we have for any function $f$ which is invariant under $\Gamma_0(N/d)$ and any integer $r>0$ that
\begin{gather}\label{eqUWrec}
f|U_{p^r}|W_p=f|U_{p^r}|B_p+\frac 1pf|U_{p^{r-1}}|W_p|B_p-\frac 1p f|U_{p^{r-1}},
\end{gather}
which is an immediate consequence of \Cref{propAL} $(i)$ and $(vi)$. By induction on $r$ this shows that we may write $P_1^{(N/d,0)}|T_{m/d}^{(N/d)}|W_p$ as a linear combination of functions of the form $P_1^{(N/d,0)}|W_p^\eps|T_{\tilde{m}}^{(N/d)}|B_d$ for $\eps\in\{0,1\}$ and suitable $\tilde{m}\in\N_0$. Since $P_1^{(N/d,0)}$ may be written as a linear combination of functions of the form $P_1^{(N,0)}|W_Q|T_Q$, the same argument as just used shows that also $P_1^{(N/d,0)}|T_{m/d}^{(N/d)}|W_p$ can be written in the form claimed in the theorem.

For a common divisor $d$ of $m$ and $N$ which is divisible by $p$, we note that we can write $W_p=\left(\frac{1}{\sqrt p}\left(\begin{smallmatrix} p & 0 \\ 0 & 1\end{smallmatrix}\right)\right)^{-1}\gamma$ for a suitable $\gamma\in\Gamma_0(N/p)$. Therefore we have that
$$P_1^{(N/d,0)}|T_{m/d}^{(N/d)}|B_d|W_p=P_1^{(N/d,0)}|T_{m/d}^{(N/d)}|B_{d/p}|\gamma=P_1^{(N/d,0)}|T_{m/d}^{(N/d)}|B_{d/p}$$
since $P_1^{(N/d,0)}|T_{m/d}^{(N/d)}|B_{d/p}\in H_0(\Gamma_0(N/p))$, so that also these summands are of the desired form. In summary, this proves the theorem for all square-free levels $N$, since the Atkin-Lehner involutions $W_p$, $p\mid N$ generate the group of all Atkin-Lehner involutions, which act transitively on the cusps.

\end{proof}

\begin{remark*}
It is essential in the above proof of \Cref{thmspan} that the Atkin-Lehner involutions $W_Q$ commute with all Hecke operators $T_m$ with $\gcd(Q,m)=1$. 
One could try to extend the the above proof to levels where the full normalizer $\Norm(\Gamma_0(N))$ acts transitively on the cusps (which is true in addition for levels $20,\ 24,\ 27,\ 32,\ 36$, but not $49$, among those where $X_0(N)$ has genus $1$). For levels $20$ and $24$ it is even true that the normalizer commutes with all Hecke operators $T_m$ with $\gcd(m,N)=1$. Unfortunately, the behaviour with respect to the $U$-operator (see \eqref{eqUWrec}) is not such that the above proof immediately generalizes. We point out however that the first part of the proof that one can express all the Poincar\'e series $P_1^{(d,k)}$ for $d\mid N$ using only elements in the normalizer and Hecke operators together with $B$-operators, works for levels $20$ and $24$.
\end{remark*}
\subsection{Proof of the dimension formulas}
We now proceed to the proof of \Cref{thmdim} and indicate the extension for composite levels.
\begin{proof}[Proof of \Cref{thmdim}]
Let $N=p$ be a prime number such that $E:=X_0(p)$ has genus $1$, i.e. $p\in\{11,17,19\}$. From \Cref{proptrans} we know that $\ch_{V^G}$ is a modular function for $\Gamma_0(p)$ without poles in $\HH$ and furthermore the transformation behaviour
\begin{gather}\label{eqchS}
\ch_{V^G}(S.\tau)=\frac{1}{p}q^{-1}\sum_{n=0} \dim V^G_{n} q^n+\frac 1pq^{-1}\sum_{i=1}^{p-1}\sum_{n=0}^\infty \dim V(\sigma^i)_{n/p}q^{n/p}.
\end{gather}
Using this together with \Cref{thmspan}, we can express $\ch_{V^G}$ solely in terms of Weierstrass mock modular forms, the Fricke involution $W_p$, and Hecke operators, more precisely we find, using that $\dim V_0=1$,
\begin{multline}
\ch_{V^G}=\Zhat_E+\frac 1p \sum_{i=1}^{p-1} \sum_{j=1}^{p-1} (p-j)\dim V(g^i)_{j/p}\Zhat_E|W_p|T_{p-j}\\
+\frac 1p\left((p\Zhat_E|U_p+\Zhat_E|W_p)|B_p+p\Zhat_E|W_p|U_p\right)+(\dim V_1^G-C),
\end{multline}
with 
\begin{multline}
C=\dim V_1^G-\left(c_E(0)+\frac {c_{E,W_p}(0)}p\sum_{i=1}^{p-1} \sum_{j=1}^{p-1}\sigma_1(p-j)\dim V(g^i)_{j/p}\right.\\
+\left.\frac{1}{p}\left[pc_E(0)+(p+1)c_{E,W_p}(0)\right]\right),
\end{multline}
where $c_E(0)=-a_E(2)/2$ and $c_{E,W_p}(0)=\widehat\zeta(\Lambda_E;L(E,1))$ denote the constant terms of $\Zhat_E$ and $\Zhat_E|W_p$ respectively, see \Cref{secWeierstrass}. Therefore we find, using the relation $S=W_pB_p^{-1}$ and the fact that both $p\Zhat_E|U_p+\Zhat_E|W_p$ and $\Zhat_E+p\Zhat_E|W_p|U_p$ transform like modular forms of level $1$ by \Cref{propAL} $(vi)$ and hence are invariant under $S$,
\begin{multline}
\ch_{V^G}|S=\Zhat_E|W_p|B_p^{-1}+\frac 1p \sum_{i=1}^{p-1} \sum_{j=1}^{p-1} (p-j)\dim V(g^i)_{j/p}\Zhat_E|T_{p-j}|B_p^{-1}\\
+\frac 1p\left(p\Zhat_E|U_p|B_p^{-1}+\Zhat_E|B_p^{-1}+p\Zhat_E|W_p|U_p|B_p^{-1}\right)+(\dim V_1^G-C).
\end{multline}
Now \Cref{propEMS} implies that
$$\ch_{V^{\orb(g^N)}}+\ch_{V^{\orb(g)}}=F_0+F_\infty,$$
for $F_\fraka$ as in \eqref{eqFfraka}. By definition we have $F_\infty =\ch_{V^G}$, so its constant term equals $\dim V_1^G$. The constant term of $F_0$ equals $p$ times that of $\ch_{V^G}|S$, since we can write $F_0=\sum_{j=0}^{p-1} \ch_{V^G}|(ST^j)$ and all summands have the same constant term. Therefore by comparing constant terms we obtain after some simplification
\begin{multline}
\dim V_1+\dim V_1^{\orb(g)}=(p+1)\dim V_1^G-(p-1)\left(c_E(0)-c_{E,W_p}(0)\right)\\
+\left(c_E(0)-c_{E,W_p}(0)\right)\sum_{i=1}^{p-1}\sum_{j=1}^{p-1} \sigma(p-j)\dim V(g^i)_{j/p}.
\end{multline}
Plugging in the definitions of $c_E(0)$ and $c_{E,W_p}(0)$, we arrive at the dimension formula stated in \Cref{thmdim}
\end{proof}
We now proceed to the proof of \Cref{thmdim2}.
\begin{proof}[Proof of \Cref{thmdim2}]
Since $\ch_{V^G}$ is a modular function, its image under the $\xi$-operator (see \Cref{propxi}) must be $0$. Hence, applying the Bruinier-Funke pairing with the newform $f$ must yield zero as well. Since we have $\ch_{V^G}(\tau)=q^{-1}+O(1)$ and the expansion at the cusp $0$ as given in \eqref{eqchS} together with $(f_E|W_p)(\tau)=\eps f(\tau)$ with $\eps\in\{\pm 1\}$, we can apply \Cref{propBF} and obtain
$$a(1)+\eps\left(\frac{a(p)}{p}+\frac 1p\sum_{i=1}^{p-1}\sum_{j=1}^{p-1} a(p-j)V(\sigma^i)_{j/p}\right)=0.$$
Since $a(1)=1$, the formula claimed in \Cref{thmdim2} follows.
\end{proof}

To conclude, we formulate the dimension formula for the square-free composite levels $N$ such that the modular curve $X_0(N)$ has genus $1$, i.e. $N\in\{14,15,21\}$. In what follows, we always write $N=p_1p_2$ where $p_1$ and $p_2$ are primes. In view of \Cref{remgeneralize}, we point out that it is not essential that $N$ has exactly two prime factors and in principle, the method of proof would go through for arbitrary square-free numbers $N$. However, for the sake of simplicity of the exposition and since all the cases under consideration here are of this form, we restrict to the situation of precisely two distinct prime factors. Before stating the result, we record the following lemma about expansions at other cusps, which is essentially the same as \cite[Corollary 2.7]{BeneishLarson} and also follows easily from \eqref{eqUWrec}.
\begin{lemma}\label{lemexpansion}
Let $N$ be any square-free number and let $f\in H_0(N)$ be a harmonic Maa{\ss} form satisfying $f(\tau)=q^{-\nu}+c+O(\exp(-\alpha y))$ and $(f|W_Q)(\tau)=c_Q+O(\exp(-\alpha y))$ for all (exact) divisors $Q\mid N$, where $\alpha>0$ and $c,c_Q\in\C$. Then if $p\mid N$ is a prime wit $p\nmid \nu$ and $a\in\Z_{>0}$, then the function $-p^{a+1}f|U_{p^{a+1}}|W_p$ has only a pole at the cusp $\infty$ and we have the expansion
$$-p^{a+1}(f|U_{p^{a+1}}|W_p)(\tau)=q^{p^a\nu}-(p-1)\sigma(p^a)c-c_p+O(\exp(-\alpha y)).$$
\end{lemma}
We now state the dimension formula for composite $N$.
\begin{theorem}\label{thmdimcomp}
With the notation and hypotheses as in \Cref{thmdim}, only that the group $G=\langle g\rangle$ has order $N=p_1p_2$, we have the following dimension formula:
\begin{gather*}
\begin{aligned}
 &\dim V^{\orb(g)}_1+\dim V^{\orb(g^{p_1})}_1 +\dim V^{\orb(g^{p_2})}_1+\dim V_1\\
&=[\SLZ:\Gamma_0(N)]\dim V^G_1+Nc_{E,N}+p_1c_{E,p_1}+p_2c_{E,p_2}\\
+&\sum_{m=0}^{N-1}\sum_{\substack{i,j\in\Z/N\Z \\ ij\equiv m\:(N)}} \sigma((N-m)')\dim W^{(i,j)}_{m/N}\left[c_{E,N}(0)+p_2^{b_m+1}c_{E,p_1}(0)\right.\\
&\qquad\quad\left.+(p_1^{a_m+1}+p^{a_m})c_{E,p_2}(0)-(p_1^{a_m+1}+p_1^{a_m}+p_2^{b_m+1}+p_2^{b_m})c_E(0)\right]\\
+&\sum_{m=0}^{p_2-1}\sum_{\substack{i,j\in\Z/p_{2}\Z \\ p_{1}ij\equiv m\:(p_2)}} \sigma((N-p_1m)')\dim W^{(p_1 i,p_1 j)}_{m/p_2}\left[\frac{p_1}{p_2}c_N+p_1\left(p_2^{b_{p_1m}+1}+p_2^{b_{p_1m}}-\frac{1}{p_2}\right)c_{E,p_1}(0)\right.\\
&\qquad\quad\left.-\left(p_1^{a_{p_1m}+2}+p_1^{a_{p_1m}+1}+\frac{p_1}{p_2}\right)c_{E,p_2}(0)+\left((p_1+1)(p_1^{a_{p_1m}+1}-p_2^{b_{p_1m}+1})+\frac{p_1}{p_2}\right)c_{E}(0)\right]\\
+&\sum_{m=0}^{p_1-1}\sum_{\substack{i,j\in\Z/p_{1}\Z \\ p_{2}ij\equiv m\:(p_1)}} \sigma((N-p_2m)')\dim W^{(p_2 i,p_2 j)}_{m/p_1}\left[\frac{p_2}{p_1}c_N-p_2\left(p_2^{b_{p_2m}+1}+p_2^{b_{p_2m}}-\frac{1}{p_1}\right)c_{E,p_1}(0)\right.\\
&\qquad\quad\left.+p_2\left(p_1^{a_{p_2m}+1}+p_1^{a_{p_2m}}-\frac{1}{p_1}\right)c_{E,p_2}(0)\right.\\
&\qquad\qquad\qquad\qquad\qquad\qquad\left.+\left((p_2+1)p_2^{b_{p_2m}+1}-p_2(p_1^{a_{p_2m}+1}+p_1^{a_{p_2m}})+\frac{p_2}{p_1}\right)c_{E}(0)\right].
\end{aligned}
\end{gather*}
where $c_E(0)=-a_E(2)/2$ denotes the constant term at infinity of $\Zhat_E$ and $c_{E,Q}(0)$ denotes the constant term of $\Zhat_E|W_Q$.
\end{theorem}
\begin{proof}
We fix our Atkin-Lehner operators as $W_{p_\ell}=\left(\begin{smallmatrix} * & * \\ p_\ell & p_{\ell\pm 1}\end{smallmatrix}\right)B_{p_\ell}$, where the first matrix, which we denote by $\gamma_{p_\ell}$, is in $\Gamma_0(p_\ell)$, and $W_N=S B_N$. By \Cref{propEMS} we obtain the following expansions of $\ch_{V^G}$ at the cusps of $\Gamma_0(N)$:
\begin{gather*}
\begin{aligned}
(\ch_{V^G}|S)(\tau)&=\frac 1N\sum_{m=0}^\infty \sum_{\substack{i,j\in\Z/N\Z \\ ij\equiv m\:(N)}} \dim W^{(i,j)}_{m/N} q^{m/N-1}\\
(\ch_{V^G}|\gamma_{p_\ell})(\tau)&=\frac{1}{p_{\ell\pm 1}}\sum_{m=0}^\infty\sum_{\substack{i,j\in\Z/p_{\ell\pm 1}\Z \\ p_{\ell}^2ij\equiv m\:(N)}} \dim W^{(p_\ell i,p_\ell j)}_{m/N} q^{m/N-1}
\end{aligned}
\end{gather*}
Note that in the formula for $\ch_{V^G}|\gamma_{p_\ell}$ only those summands survive where $m$ is divisible by $p_\ell$.

Using \Cref{corpoincare}, \Cref{lemHeckePoincare}, and \Cref{lemexpansion} we can express the unique (up to additive constants) harmonic Maa{\ss} form in $H_0(N)$ having a pole of order $m=p_1^ap_2^bm'$, $\gcd(N,m')=1$, at $\infty$ and nowhere else as
\begin{multline}
p_1p_2m(\Zhat_E|T_{m'}|U_{p_1^{a+1}}|W_{p_1}|U_{p_2^{b+1}}|W_{p_2})(\tau)\\
=q^{-m}+\sigma(m')\left((p_1^{a+1}-1)(p_2^{b+1}-1)c_E(0)+(p_2^{b+1}-1)c_{E,p_1}(0)\right.\\
\left.+(p_1^{a+1}-1)c_{E,p_2}(0)+c_{E,N}(0)\right)
+O(\exp(-\alpha y))
\end{multline}
for some suitable $\alpha>0$ and $c_E(0)$ and $c_{E,Q}(0)$ denoting the constant terms of $\Zhat_E$ and $\Zhat_E|W_Q$, respectively. For reference, we also give the following expansions which are needed in order to express $\ch_{V^G}$ in terms of Weierstrass mock modular forms,
\begin{gather}
p_1p_2m(\Zhat_E|T_{m'}|U_{p_1^{a+1}}|U_{p_2^{b+1}})(\tau)
=p_1^{a+1}p_2^{b+1}\sigma(m')c_E(0)
+O(\exp(-\alpha y)),
\end{gather}
\begin{multline}
p_1p_2m(\Zhat_E|T_{m'}|U_{p_1^{a+1}}|U_{p_2^{b+1}}|W_{p_2})(\tau)\\
=p_1^{a+1}\sigma(m')\left((p_2^{b+1}-1)c_E(0)+c_{E,p_2}(0)\right)
+O(\exp(-\alpha y)),
\end{multline}
and
\begin{multline}
p_1p_2m(\Zhat_E|T_{m'}|U_{p_1^{a+1}}|W_{p_1}|U_{p_2^{b+1}})(\tau)\\
=p_2^{b+1}\sigma(m')\left((p_1^{a+1}-1)c_E(0)+c_{E,p_1}(0)\right)
+O(\exp(-\alpha y)).
\end{multline}
Hence we may write
\begin{gather}\label{eqchVGcomp}
\begin{aligned}
\ch_{V^G}&=\Zhat_E+\frac 1N \sum_{m=0}^{N-1}\sum_{\substack{i,j\in\Z/N\Z \\ ij\equiv m\:(N)}} \dim W^{(i,j)}_{m/N}(-p_2^{b_m+1}(-p_1^{a_m+1}((N-m)'\Zhat_E|T_{(N-m)'})\\
&\qquad\qquad\qquad\qquad\qquad|U_{p_1^{a_m+1}}|W_{p_1})|U_{p_2^{b_m+1}}|W_{p_2})|W_N\\
&+\frac{1}{p_2}\sum_{m=0}^{p_2-1}\sum_{\substack{i,j\in\Z/p_{2}\Z \\ p_{1}ij\equiv m\:(p_2)}} \dim W^{(p_1 i,p_1 j)}_{m/p_2}(-p_2^{b_{p_1m}+1}(-p_1^{a_{p_1m}+1}((N-p_1m)'\Zhat_E|T_{(N-p_1m)'})\\
&\qquad\qquad\qquad\qquad\qquad|U_{p_1^{a_{p_1m}+1}}|W_{p_1})|U_{p_2^{b_{p_1m}+1}}|W_{p_2})|W_{p_1}\\
&+p_2\sum_{m=0}^{p_1-1}\sum_{\substack{i,j\in\Z/p_{1}\Z \\ p_{2}ij\equiv m\:(p_1)}} \dim W^{(p_2 i,p_2 j)}_{m/p_1} (-p_2^{b_{p_2m}+1}(-p_1^{a_{p_2m}+1}((N-p_2m)'\Zhat_E|T_{(N-p_2m)'})\\
&\qquad\qquad\qquad\qquad\qquad|U_{p_1^{a_{p_2m}+1}}|W_{p_1})|U_{p_2^{b_{p_2m}+1}}|W_{p_2})|W_{p_2}\\
&+\dim V_1^G-C
\end{aligned}
\end{gather}
where we write $N-m=p_1^{a_m}p_2^{b_m}(N-m)'$ with $\gcd(N,(N-m)')=1$ and the constant $C$ is defined by
\begin{gather}\label{eqchVGconst}
\begin{aligned}
C&=c_E(0)+\frac 1N\sum_{m=1}^{N-1}\sum_{\substack{i,j\in\Z/N\Z \\ ij\equiv m\:(N)}} \dim W^{(i,j)}_{m/N} p_1^{a_m+1}p_2^{b_m+1}\sigma((N-m)')c_E(0)\\
&+\frac{1}{p_2}\sum_{m=0}^{p_2-1}\sum_{\substack{i,j\in\Z/p_{2}\Z \\ p_{1}ij\equiv m\:(p_2)}} \dim W^{(p_1 i,p_1 j)}_{m/p_2} p_1^{a_{p_1m}+1}\sigma((N-p_1m)')\left[(p_2^{b_{p_1m}+1}-1)c_E(0)+c_{E,p_2}(0)\right]\\
&+\frac{1}{p_1}\sum_{m=0}^{p_1-1}\sum_{\substack{i,j\in\Z/p_{1}\Z \\ p_{2}ij\equiv m\:(p_1)}} \dim W^{(p_2 i,p_2 j)}_{m/p_1} p_2^{b_{p_2m}+1}\sigma((N-p_2m)')\left[(p_1^{a_{p_2m}+1}-1)c_E(0)+c_{E,p_1}(0)\right].
\end{aligned}
\end{gather}
As in the proof of \Cref{thmdim}, we compute also the expansions of $\ch_{V^G}$ at all other cusps, which is fairly straightforward from the expression in \eqref{eqchVGcomp} using once more the known commutation relations among Hecke and Atkin-Lehner operators in \Cref{propAL}, so we refrain from giving these expansions explicitly for the sake of brevity. 

By \Cref{propEMS}, we have that
$$
\sum_{\fraka} F_\fraka(\tau)=\ch_{V^{\orb(g)}}(\tau)+\ch_{V^{\orb(g^{p_1})}}(\tau)+\ch_{V^{\orb(g^{p_2})}}(\tau)+\ch_{V}(\tau),
$$
where the sum runs over a complete set of representatives of cusps of $\Gamma_0(N)$, which we may and do fix as $\{1/N(\equiv \infty),1/p_1,1/p_2,1(\equiv 0)\}$. It is easy to see that the constant term of $F_\fraka$ is precisely the constant term of $\ch_{V^G}|W_{N\fraka}$ multiplied by the width of the cusp, which in this case is $N\fraka$. Using this observation, we obtain the dimension formula stated in the theorem after some simplification steps.
\end{proof}

\end{document}